\documentclass{amsart}
\usepackage{amssymb}
\usepackage{amsmath}
\usepackage{amsfonts}

\newtheorem{theorem}{Theorem}
\theoremstyle{plain}

\newtheorem{lemma}{Lemma}

\numberwithin{equation}{section}
\input{tcilatex}

\begin{document}
\title[ON\ THE SIMPSON'S\ INEQUALITY]{ON\ THE SIMPSON'S\ INEQUALITY FOR\
CO-ORDINATED\ CONVEX\ FUNCTIONS}
\author{M. Emin \"{O}zdemir$^{\blacklozenge }$}
\address{$^{\blacklozenge }$Ataturk University, K.K. Education Faculty,
Department of Mathematics, 25240, Erzurum, Turkey}
\email{emos@atauni.edu.tr}
\author{Ahmet Ocak Akdemir$^{\spadesuit ,\bigstar }$}
\address{$^{\spadesuit }$A\u{g}r\i\ \.{I}brahim \c{C}e\c{c}en University,
Faculty of Science and Arts, Department of Mathematics, 04100, A\u{g}r\i ,
Turkey}
\email{ahmetakdemir@agri.edu.tr}
\author{Havva Kavurmac\i $^{\blacklozenge }$}
\address{$^{\blacklozenge }$Ataturk University, K.K. Education Faculty,
Department of Mathematics, 25240, Erzurum, Turkey}
\email{hkavurmaci@atauni.edu.tr}
\author{Merve Avc\i $^{\blacklozenge }$}
\email{merveavci@ymail.com}
\thanks{$^{\bigstar }$Corresponding Author}
\date{December 29, 2010}

\begin{abstract}
In this paper, a new lemma is proved and inequalities of Simpson type are
established for co-ordinated convex functions and bounded functions.
\end{abstract}

\maketitle

\section{INTRODUCTION}

The following inequality is well-known in the literature as Simpson's
inequality:

\begin{theorem}
Let $f:\left[ a,b\right] \rightarrow 
\mathbb{R}
$ be a four times continuously differentiable mapping on $\left[ a,b\right] $
and $\left\Vert f^{\left( 4\right) }\right\Vert _{\infty }=\sup\limits_{x\in %
\left[ a,b\right] }\left\vert f^{\left( 4\right) }\left( x\right)
\right\vert <\infty .$ Then the folllowing inequality holds:%
\begin{equation*}
\left\vert \frac{1}{3}\left[ \frac{f\left( a\right) +f\left( b\right) }{2}%
+2f\left( \frac{a+b}{2}\right) \right] -\frac{1}{b-a}\int_{a}^{b}f\left(
x\right) dx\right\vert \leq \frac{1}{2880}\left\Vert f^{\left( 4\right)
}\right\Vert _{\infty }\left( b-a\right) ^{4}.
\end{equation*}
\end{theorem}

For recent results on Simpson's type inequalities see the papers \cite{SOZ1}%
, \cite{SOZ2}, \cite{SOZ3}, \cite{LIU}, \cite{SS2}, \cite{AL1}, \cite{UJ}
and \cite{LU}.

In \cite{SS}, Dragomir defined co-ordinated convexity and proved the
following inequalities:

Let us consider the bidimensional interval $\Delta =\left[ a,b\right] \times %
\left[ c,d\right] $ in $%
\mathbb{R}
^{2}$ with $a<b$ and $c<d.$ A function $f:\Delta \rightarrow 
\mathbb{R}
$ will be called convex on the co-ordinates if the partial mappings%
\begin{equation*}
f_{y}:\left[ a,b\right] \rightarrow 
\mathbb{R}
,\text{ \ \ }f_{y}\left( u\right) =f\left( u,y\right)
\end{equation*}%
and%
\begin{equation*}
f_{x}:\left[ c,d\right] \rightarrow 
\mathbb{R}
,\text{ \ \ }f_{x}\left( v\right) =f\left( x,v\right)
\end{equation*}%
are convex where defined for all $y\in \left[ c,d\right] $ and $x\in \left[
a,b\right] .$

Recall that the mapping $f:\Delta \rightarrow 
\mathbb{R}
$ is convex on $\Delta $, if the following inequality;%
\begin{equation*}
f\left( \lambda x+\left( 1-\lambda \right) z,\lambda y+\left( 1-\lambda
\right) w\right) \leq \lambda f\left( x,y\right) +\left( 1-\lambda \right)
f\left( z,w\right)
\end{equation*}%
holds for all $\left( x,y\right) ,$ $\left( z,w\right) \in \Delta $ and $%
\lambda \in \left[ 0,1\right] .$

\begin{theorem}
Suppose that $f:\Delta =\left[ a,b\right] \times \left[ c,d\right]
\rightarrow 
\mathbb{R}
$ is convex on the co-ordinates on $\Delta .$ Then one has the inequalities;%
\begin{eqnarray}
&&f\left( \frac{a+b}{2},\frac{c+d}{2}\right)  \notag \\
&\leq &\frac{1}{\left( b-a\right) \left( d-c\right) }\int_{a}^{b}%
\int_{c}^{d}f\left( x,y\right) dydx  \label{1.1} \\
&\leq &\frac{f\left( a,c\right) +f\left( b,c\right) +f\left( a,d\right)
+f\left( b,d\right) }{4}  \notag
\end{eqnarray}%
The above inequalities are sharp.
\end{theorem}

Recently, several papers have been written on co-ordinated convexity. In 
\cite{HTS}, Hwang et al. gave a refinement of Hadamard's inequality on the
co-ordinates and they proved some inequalities for co-ordinated convex
functions. In \cite{DAR1}, \cite{DAR2} and \cite{DAR3}, Alomari and Darus
proved inequalities for co-ordinated $s-$convex functions. In \cite{AH},
Akdemir and \"{O}zdemir gave definition of co-ordinated $P-$convex functions
and Godunova-Levin functions and proved some inequalities. In \cite{AL},
Latif and Alomari defined co-ordinated $h-$convex functions and established
some inequalities for co-ordinated $h-$convex functions. In \cite{LAT},
Latif and Alomari proved inequalities involving product of convex functions
on the co-ordinates. In \cite{MSET}, \"{O}zdemir et al. defined co-ordinated 
$m-$convexity and $\left( \alpha ,m\right) -$convexity and gave
inequalities. In \cite{SET3}, Sar\i kaya et al. proved a new lemma and
established some inequalities on co-ordinated convex functions. In \cite{BAK}%
, Bakula and Pecaric gave several Jensen-type inequalities for co-ordinated
convex functions.

In this paper, we give a Simpson-type inequality for co-ordinated convex
functions on the basis of the following lemma.

\section{MAIN\ RESULTS}

To prove our main result, we need the following lemma.

\begin{lemma}
Let $f:\Delta \subset 
\mathbb{R}
^{2}\rightarrow 
\mathbb{R}
$ be a partial differentiable mapping on $\Delta =\left[ a,b\right] \times %
\left[ c,d\right] .$ If $\frac{\partial ^{2}f}{\partial t\partial s}\in
L\left( \Delta \right) ,$ then the following equality holds:%
\begin{eqnarray}
&&\frac{f\left( a,\frac{c+d}{2}\right) +f\left( b,\frac{c+d}{2}\right)
+4f\left( \frac{a+b}{2},\frac{c+d}{2}\right) +f\left( \frac{a+b}{2},c\right)
+f\left( \frac{a+b}{2},d\right) }{9}  \notag \\
&&+\frac{f\left( a,c\right) +f\left( b,c\right) +f\left( a,d\right) +f\left(
b,d\right) }{36}  \label{2.1} \\
&&-\frac{1}{6\left( b-a\right) }\int_{a}^{b}\left[ f\left( x,c\right)
+4f\left( x,\frac{c+d}{2}\right) +f\left( x,d\right) \right] dx  \notag \\
&&-\frac{1}{6\left( d-c\right) }\int_{c}^{d}\left[ f\left( a,y\right)
+4f\left( \frac{a+b}{2},y\right) +f\left( b,y\right) \right] dy  \notag \\
&&+\frac{1}{\left( b-a\right) \left( d-c\right) }\int_{a}^{b}\int_{c}^{d}f%
\left( x,y\right) dydx  \notag \\
&=&\left( b-a\right) \left( d-c\right) \int_{0}^{1}\int_{0}^{1}p\left(
x,t\right) q\left( y,s\right) \frac{\partial ^{2}f}{\partial t\partial s}%
\left( ta+\left( 1-t\right) b,sc+\left( 1-s\right) d\right) dtds  \notag
\end{eqnarray}%
where%
\begin{equation*}
p\left( x,t\right) =\left\{ 
\begin{array}{c}
\left( t-\frac{1}{6}\right) ,\text{ \ \ \ \ }t\in \left[ 0,\frac{1}{2}\right]
\\ 
\\ 
\left( t-\frac{5}{6}\right) ,\text{ \ \ \ \ }t\in \left( \frac{1}{2},1\right]%
\end{array}%
\right.
\end{equation*}%
and%
\begin{equation*}
q\left( y,s\right) =\left\{ 
\begin{array}{c}
\left( s-\frac{1}{6}\right) ,\text{ \ \ \ \ }s\in \left[ 0,\frac{1}{2}\right]
\\ 
\\ 
\left( s-\frac{5}{6}\right) ,\text{ \ \ \ \ }s\in \left( \frac{1}{2},1\right]%
\end{array}%
\right. .
\end{equation*}
\end{lemma}

\begin{proof}
Integrating by parts, we can write%
\begin{eqnarray*}
&&\int_{0}^{1}\int_{0}^{1}p\left( x,t\right) q\left( y,s\right) \frac{%
\partial ^{2}f}{\partial t\partial s}\left( ta+\left( 1-t\right) b,sc+\left(
1-s\right) d\right) dtds \\
&=&\int_{0}^{1}q\left( y,s\right) \left[ \int_{0}^{\frac{1}{2}}\left( t-%
\frac{1}{6}\right) \frac{\partial ^{2}f}{\partial t\partial s}\left(
ta+\left( 1-t\right) b,sc+\left( 1-s\right) d\right) dt\right. \\
&&\left. +\int_{\frac{1}{2}}^{1}\left( t-\frac{5}{6}\right) \frac{\partial
^{2}f}{\partial t\partial s}\left( ta+\left( 1-t\right) b,sc+\left(
1-s\right) d\right) dt\right] ds
\end{eqnarray*}%
By integrating the right hand side of equality, we get%
\begin{eqnarray*}
&&\int_{0}^{1}q\left( y,s\right) \left\{ \left[ \left( t-\frac{1}{6}\right)
\left( \frac{1}{a-b}\right) \frac{\partial f}{\partial s}\left( ta+\left(
1-t\right) b,sc+\left( 1-s\right) d\right) \right] _{0}^{\frac{1}{2}}\right.
\\
&&-\frac{1}{a-b}\int_{0}^{\frac{1}{2}}\frac{\partial f}{\partial s}\left(
ta+\left( 1-t\right) b,sc+\left( 1-s\right) d\right) dt \\
&&+\left[ \left( t-\frac{5}{6}\right) \left( \frac{1}{a-b}\right) \frac{%
\partial f}{\partial s}\left( ta+\left( 1-t\right) b,sc+\left( 1-s\right)
d\right) \right] _{\frac{1}{2}}^{1} \\
&&\left. -\frac{1}{a-b}\int_{\frac{1}{2}}^{1}\frac{\partial f}{\partial s}%
\left( ta+\left( 1-t\right) b,sc+\left( 1-s\right) d\right) dt\right\} ds
\end{eqnarray*}%
\begin{eqnarray*}
&=&\frac{1}{b-a}\left\{ -\frac{1}{3}\int_{0}^{\frac{1}{2}}\left( s-\frac{1}{6%
}\right) \frac{\partial f}{\partial s}\left( \frac{a+b}{2},sc+\left(
1-s\right) d\right) ds\right. \\
&&-\frac{1}{3}\int_{\frac{1}{2}}^{1}\left( s-\frac{5}{6}\right) \frac{%
\partial f}{\partial s}\left( \frac{a+b}{2},sc+\left( 1-s\right) d\right) ds
\\
&&-\frac{1}{6}\int_{0}^{\frac{1}{2}}\left( s-\frac{1}{6}\right) \frac{%
\partial f}{\partial s}\left( b,sc+\left( 1-s\right) d\right) ds \\
&&-\frac{1}{6}\int_{\frac{1}{2}}^{1}\left( s-\frac{5}{6}\right) \frac{%
\partial f}{\partial s}\left( b,sc+\left( 1-s\right) d\right) ds \\
&&+\int_{0}^{\frac{1}{2}}\int_{0}^{\frac{1}{2}}\left( s-\frac{1}{6}\right) 
\frac{\partial f}{\partial s}\left( ta+\left( 1-t\right) b,sc+\left(
1-s\right) d\right) dsdt \\
&&+\int_{\frac{1}{2}}^{1}\int_{0}^{\frac{1}{2}}\left( s-\frac{5}{6}\right) 
\frac{\partial f}{\partial s}\left( ta+\left( 1-t\right) b,sc+\left(
1-s\right) d\right) dsdt \\
&&-\frac{1}{6}\int_{0}^{\frac{1}{2}}\left( s-\frac{1}{6}\right) \frac{%
\partial f}{\partial s}\left( a,sc+\left( 1-s\right) d\right) ds \\
&&-\frac{1}{6}\int_{\frac{1}{2}}^{1}\left( s-\frac{5}{6}\right) \frac{%
\partial f}{\partial s}\left( a,sc+\left( 1-s\right) d\right) ds \\
&&-\frac{1}{3}\int_{0}^{\frac{1}{2}}\left( s-\frac{1}{6}\right) \frac{%
\partial f}{\partial s}\left( \frac{a+b}{2},sc+\left( 1-s\right) d\right) ds
\\
&&-\frac{1}{3}\int_{\frac{1}{2}}^{1}\left( s-\frac{5}{6}\right) \frac{%
\partial f}{\partial s}\left( \frac{a+b}{2},sc+\left( 1-s\right) d\right) ds
\\
&&+\int_{0}^{\frac{1}{2}}\int_{\frac{1}{2}}^{1}\left( s-\frac{1}{6}\right) 
\frac{\partial f}{\partial s}\left( ta+\left( 1-t\right) b,sc+\left(
1-s\right) d\right) dsdt \\
&&\left. +\int_{\frac{1}{2}}^{1}\int_{\frac{1}{2}}^{1}\left( s-\frac{5}{6}%
\right) \frac{\partial f}{\partial s}\left( ta+\left( 1-t\right) b,sc+\left(
1-s\right) d\right) dsdt\right\} .
\end{eqnarray*}%
Computing these integrals and using the change of the variable $x=ta+\left(
1-t\right) b$ and $y=sc+\left( 1-s\right) d$ for $\left( t,s\right) \in %
\left[ 0,1\right] ^{2},$ then multiplying both sides with $\left( b-a\right)
\left( d-c\right) ,$ we get the desired result.
\end{proof}

\begin{theorem}
Let $f:\Delta \subset 
\mathbb{R}
^{2}\rightarrow 
\mathbb{R}
$ be a partial differentiable mapping on $\Delta =\left[ a,b\right] \times %
\left[ c,d\right] .$ If $\frac{\partial ^{2}f}{\partial t\partial s}$ is a
convex function on the co-ordinates on $\Delta ,$ then the following
inequality holds:%
\begin{eqnarray*}
&&\left\vert \frac{f\left( a,\frac{c+d}{2}\right) +f\left( b,\frac{c+d}{2}%
\right) +4f\left( \frac{a+b}{2},\frac{c+d}{2}\right) +f\left( \frac{a+b}{2}%
,c\right) +f\left( \frac{a+b}{2},d\right) }{9}\right. \\
&&+\frac{f\left( a,c\right) +f\left( b,c\right) +f\left( a,d\right) +f\left(
b,d\right) }{36}\left. +\frac{1}{\left( b-a\right) \left( d-c\right) }%
\int_{a}^{b}\int_{c}^{d}f\left( x,y\right) dydx-A\right\vert \\
&\leq &\frac{25\left( b-a\right) \left( d-c\right) }{72} \\
&&\times \frac{\left\vert \frac{\partial ^{2}f}{\partial t\partial s}\left(
a,c\right) \right\vert +\left\vert \frac{\partial ^{2}f}{\partial t\partial s%
}\left( a,d\right) \right\vert +\left\vert \frac{\partial ^{2}f}{\partial
t\partial s}\left( b,c\right) \right\vert +\left\vert \frac{\partial ^{2}f}{%
\partial t\partial s}\left( b,d\right) \right\vert }{72}
\end{eqnarray*}%
where%
\begin{eqnarray*}
A &=&\frac{1}{6\left( b-a\right) }\int_{a}^{b}\left[ f\left( x,c\right)
+4f\left( x,\frac{c+d}{2}\right) +f\left( x,d\right) \right] dx \\
&&+\frac{1}{6\left( d-c\right) }\int_{c}^{d}\left[ f\left( a,y\right)
+4f\left( \frac{a+b}{2},y\right) +f\left( b,y\right) \right] dy.
\end{eqnarray*}
\end{theorem}

\begin{proof}
By using Lemma 1, we can write 
\begin{eqnarray*}
&&\left\vert \frac{f\left( a,\frac{c+d}{2}\right) +f\left( b,\frac{c+d}{2}%
\right) +4f\left( \frac{a+b}{2},\frac{c+d}{2}\right) +f\left( \frac{a+b}{2}%
,c\right) +f\left( \frac{a+b}{2},d\right) }{9}\right. \\
&&\left. +\frac{f\left( a,c\right) +f\left( b,c\right) +f\left( a,d\right)
+f\left( b,d\right) }{36}+\frac{1}{\left( b-a\right) \left( d-c\right) }%
\int_{a}^{b}\int_{c}^{d}f\left( x,y\right) dydx-A\right\vert \\
&\leq &\left( b-a\right) \left( d-c\right)
\int_{0}^{1}\int_{0}^{1}\left\vert p\left( x,t\right) q\left( y,s\right)
\right\vert \left\vert \frac{\partial ^{2}f}{\partial t\partial s}\left(
ta+\left( 1-t\right) b,sc+\left( 1-s\right) d\right) \right\vert dtds.
\end{eqnarray*}%
Since $f:\Delta \rightarrow 
\mathbb{R}
$ is co-ordinated convex on $\Delta ,$ we get%
\begin{eqnarray*}
&&\left\vert \frac{f\left( a,\frac{c+d}{2}\right) +f\left( b,\frac{c+d}{2}%
\right) +4f\left( \frac{a+b}{2},\frac{c+d}{2}\right) +f\left( \frac{a+b}{2}%
,c\right) +f\left( \frac{a+b}{2},d\right) }{9}\right. \\
&&\left. +\frac{f\left( a,c\right) +f\left( b,c\right) +f\left( a,d\right)
+f\left( b,d\right) }{36}+\frac{1}{\left( b-a\right) \left( d-c\right) }%
\int_{a}^{b}\int_{c}^{d}f\left( x,y\right) dydx-A\right\vert \\
&\leq &\left( b-a\right) \left( d-c\right) \int_{0}^{1}\left\vert q\left(
y,s\right) \right\vert \left[ \int_{0}^{1}\left\vert p\left( x,t\right)
\right\vert \left\{ t\left\vert \frac{\partial ^{2}f}{\partial t\partial s}%
\left( a,sc+\left( 1-s\right) d\right) \right\vert \right. \right. \\
&&\left. +\left. \left( 1-t\right) \left\vert \frac{\partial ^{2}f}{\partial
t\partial s}\left( b,sc+\left( 1-s\right) d\right) \right\vert dt\right\} %
\right] ds.
\end{eqnarray*}%
Computing the integral in the right hand side of above inequality, we have%
\begin{eqnarray*}
&&\int_{0}^{1}\left\vert p\left( x,t\right) \right\vert \left\{ t\left\vert 
\frac{\partial ^{2}f}{\partial t\partial s}\left( a,sc+\left( 1-s\right)
d\right) \right\vert +\left( 1-t\right) \left\vert \frac{\partial ^{2}f}{%
\partial t\partial s}\left( b,sc+\left( 1-s\right) d\right) \right\vert
\right\} dt \\
&=&\int_{0}^{\frac{1}{6}}\left( \frac{1}{6}-t\right) \left\{ t\left\vert 
\frac{\partial ^{2}f}{\partial t\partial s}\left( a,sc+\left( 1-s\right)
d\right) \right\vert +\left( 1-t\right) \left\vert \frac{\partial ^{2}f}{%
\partial t\partial s}\left( b,sc+\left( 1-s\right) d\right) \right\vert
\right\} dt \\
&&+\int_{\frac{1}{6}}^{\frac{1}{2}}\left( t-\frac{1}{6}\right) \left\{
t\left\vert \frac{\partial ^{2}f}{\partial t\partial s}\left( a,sc+\left(
1-s\right) d\right) \right\vert +\left( 1-t\right) \left\vert \frac{\partial
^{2}f}{\partial t\partial s}\left( b,sc+\left( 1-s\right) d\right)
\right\vert \right\} dt \\
&&+\int_{\frac{1}{2}}^{\frac{5}{6}}\left( \frac{5}{6}-t\right) \left\{
t\left\vert \frac{\partial ^{2}f}{\partial t\partial s}\left( a,sc+\left(
1-s\right) d\right) \right\vert +\left( 1-t\right) \left\vert \frac{\partial
^{2}f}{\partial t\partial s}\left( b,sc+\left( 1-s\right) d\right)
\right\vert \right\} dt \\
&&+\int_{\frac{5}{6}}^{1}\left( t-\frac{5}{6}\right) \left\{ t\left\vert 
\frac{\partial ^{2}f}{\partial t\partial s}\left( a,sc+\left( 1-s\right)
d\right) \right\vert +\left( 1-t\right) \left\vert \frac{\partial ^{2}f}{%
\partial t\partial s}\left( b,sc+\left( 1-s\right) d\right) \right\vert
\right\} dt
\end{eqnarray*}%
\begin{equation*}
=\frac{5}{72}\left( \left\vert \frac{\partial ^{2}f}{\partial t\partial s}%
\left( a,sc+\left( 1-s\right) d\right) \right\vert +\left\vert \frac{%
\partial ^{2}f}{\partial t\partial s}\left( b,sc+\left( 1-s\right) d\right)
\right\vert \right) .
\end{equation*}%
We obtain%
\begin{eqnarray}
&&\left\vert \frac{f\left( a,\frac{c+d}{2}\right) +f\left( b,\frac{c+d}{2}%
\right) +4f\left( \frac{a+b}{2},\frac{c+d}{2}\right) +f\left( \frac{a+b}{2}%
,c\right) +f\left( \frac{a+b}{2},d\right) }{9}\right.  \label{2.3} \\
&&+\frac{f\left( a,c\right) +f\left( b,c\right) +f\left( a,d\right) +f\left(
b,d\right) }{36}  \notag \\
&&\left. +\frac{1}{\left( b-a\right) \left( d-c\right) }\int_{a}^{b}%
\int_{c}^{d}f\left( x,y\right) dydx-A\right\vert  \notag \\
&\leq &\frac{5\left( b-a\right) \left( d-c\right) }{72}  \notag \\
&&\times \int_{0}^{1}\left\vert q\left( y,s\right) \right\vert \left\{
\left\vert \frac{\partial ^{2}f}{\partial t\partial s}\left( a,sc+\left(
1-s\right) d\right) \right\vert +\left\vert \frac{\partial ^{2}f}{\partial
t\partial s}\left( b,sc+\left( 1-s\right) d\right) \right\vert \right\} ds 
\notag
\end{eqnarray}%
By a similar argument for the above integral, we have%
\begin{eqnarray}
&&\int_{0}^{1}\left\vert q\left( y,s\right) \right\vert \left\{ \left\vert 
\frac{\partial ^{2}f}{\partial t\partial s}\left( a,sc+\left( 1-s\right)
d\right) \right\vert +\left\vert \frac{\partial ^{2}f}{\partial t\partial s}%
\left( b,sc+\left( 1-s\right) d\right) \right\vert \right\} ds  \notag \\
&=&\int_{0}^{\frac{1}{6}}\left( \frac{1}{6}-s\right) \left\{ s\left\vert 
\frac{\partial ^{2}f}{\partial t\partial s}\left( a,c\right) \right\vert
+\left( 1-s\right) \left\vert \frac{\partial ^{2}f}{\partial t\partial s}%
\left( a,d\right) \right\vert \right\} ds  \label{2.4} \\
&&+\int_{0}^{\frac{1}{6}}\left( \frac{1}{6}-s\right) \left\{ s\left\vert 
\frac{\partial ^{2}f}{\partial t\partial s}\left( b,c\right) \right\vert
+\left( 1-s\right) \left\vert \frac{\partial ^{2}f}{\partial t\partial s}%
\left( b,d\right) \right\vert \right\} ds  \notag \\
&&+\int_{\frac{1}{6}}^{\frac{1}{2}}\left( s-\frac{1}{6}\right) \left\{
s\left\vert \frac{\partial ^{2}f}{\partial t\partial s}\left( a,c\right)
\right\vert +\left( 1-s\right) \left\vert \frac{\partial ^{2}f}{\partial
t\partial s}\left( a,d\right) \right\vert \right\} ds  \notag \\
&&+\int_{\frac{1}{6}}^{\frac{1}{2}}\left( s-\frac{1}{6}\right) \left\{
s\left\vert \frac{\partial ^{2}f}{\partial t\partial s}\left( b,c\right)
\right\vert +\left( 1-s\right) \left\vert \frac{\partial ^{2}f}{\partial
t\partial s}\left( b,d\right) \right\vert \right\} ds  \notag \\
&&+\int_{\frac{1}{2}}^{\frac{5}{6}}\left( \frac{5}{6}-s\right) \left\{
s\left\vert \frac{\partial ^{2}f}{\partial t\partial s}\left( a,c\right)
\right\vert +\left( 1-s\right) \left\vert \frac{\partial ^{2}f}{\partial
t\partial s}\left( a,d\right) \right\vert \right\} ds  \notag \\
&&+\int_{\frac{1}{2}}^{\frac{5}{6}}\left( \frac{5}{6}-s\right) \left\{
s\left\vert \frac{\partial ^{2}f}{\partial t\partial s}\left( b,c\right)
\right\vert +\left( 1-s\right) \left\vert \frac{\partial ^{2}f}{\partial
t\partial s}\left( b,d\right) \right\vert \right\} ds  \notag \\
&&+\int_{\frac{5}{6}}^{1}\left( s-\frac{5}{6}\right) \left\{ s\left\vert 
\frac{\partial ^{2}f}{\partial t\partial s}\left( a,c\right) \right\vert
+\left( 1-s\right) \left\vert \frac{\partial ^{2}f}{\partial t\partial s}%
\left( a,d\right) \right\vert \right\} ds  \notag \\
&&+\int_{\frac{5}{6}}^{1}\left( s-\frac{5}{6}\right) \left\{ s\left\vert 
\frac{\partial ^{2}f}{\partial t\partial s}\left( b,c\right) \right\vert
+\left( 1-s\right) \left\vert \frac{\partial ^{2}f}{\partial t\partial s}%
\left( b,d\right) \right\vert \right\} ds  \notag \\
&=&\frac{5\left[ \left\vert \frac{\partial ^{2}f}{\partial t\partial s}%
\left( a,c\right) \right\vert +\left\vert \frac{\partial ^{2}f}{\partial
t\partial s}\left( a,d\right) \right\vert +\left\vert \frac{\partial ^{2}f}{%
\partial t\partial s}\left( b,c\right) \right\vert +\left\vert \frac{%
\partial ^{2}f}{\partial t\partial s}\left( b,d\right) \right\vert \right] }{%
72}  \notag
\end{eqnarray}%
If we use (\ref{2.4}) in (\ref{2.3}), we get the required result.
\end{proof}

\begin{theorem}
Let $f:\Delta \subset 
\mathbb{R}
^{2}\rightarrow 
\mathbb{R}
$ be a partial differentiable mapping on $\Delta =\left[ a,b\right] \times %
\left[ c,d\right] .$ If $\frac{\partial ^{2}f}{\partial t\partial s}$ is
bounded, i.e.,%
\begin{equation*}
\left\Vert \frac{\partial ^{2}f}{\partial t\partial s}\left(
ta+(1-t)b,sc+(1-s)d\right) \right\Vert _{\infty }=\underset{(x,y)\in \left(
a,b\right) \times \left( c,d\right) }{\sup }\left\vert \frac{\partial ^{2}f}{%
\partial t\partial s}\left( ta+(1-t)b,sc+(1-s)d\right) \right\vert <\infty
\end{equation*}%
for all $\left( t,s\right) \in \left[ 0,1\right] ^{2}.$ Then the following
inequality holds:%
\begin{eqnarray*}
&&\left\vert \frac{f\left( a,\frac{c+d}{2}\right) +f\left( b,\frac{c+d}{2}%
\right) +4f\left( \frac{a+b}{2},\frac{c+d}{2}\right) +f\left( \frac{a+b}{2}%
,c\right) +f\left( \frac{a+b}{2},d\right) }{9}\right. \\
&&+\frac{f\left( a,c\right) +f\left( b,c\right) +f\left( a,d\right) +f\left(
b,d\right) }{36} \\
&&\left. +\frac{1}{\left( b-a\right) \left( d-c\right) }\int_{a}^{b}%
\int_{c}^{d}f\left( x,y\right) dydx-A\right\vert \\
&\leq &\frac{25\left( b-a\right) \left( d-c\right) }{1296}\left\Vert \frac{%
\partial ^{2}f}{\partial t\partial s}\left( ta+(1-t)b,sc+(1-s)d\right)
\right\Vert _{\infty }.
\end{eqnarray*}
\end{theorem}

\begin{proof}
From Lemma 1 and using the property of modulus, we have%
\begin{eqnarray*}
&&\left\vert \frac{f\left( a,\frac{c+d}{2}\right) +f\left( b,\frac{c+d}{2}%
\right) +4f\left( \frac{a+b}{2},\frac{c+d}{2}\right) +f\left( \frac{a+b}{2}%
,c\right) +f\left( \frac{a+b}{2},d\right) }{9}\right. \\
&&\left. +\frac{f\left( a,c\right) +f\left( b,c\right) +f\left( a,d\right)
+f\left( b,d\right) }{36}+\frac{1}{\left( b-a\right) \left( d-c\right) }%
\int_{a}^{b}\int_{c}^{d}f\left( x,y\right) dydx-A\right\vert \\
&\leq &\left( b-a\right) \left( d-c\right)
\int_{0}^{1}\int_{0}^{1}\left\vert p\left( x,t\right) q\left( y,s\right)
\right\vert \left\vert \frac{\partial ^{2}f}{\partial t\partial s}\left(
ta+\left( 1-t\right) b,sc+\left( 1-s\right) d\right) \right\vert dtds.
\end{eqnarray*}%
Since $\frac{\partial ^{2}f}{\partial t\partial s}$ is bounded, we have 
\begin{eqnarray}
&&\left\vert \frac{f\left( a,\frac{c+d}{2}\right) +f\left( b,\frac{c+d}{2}%
\right) +4f\left( \frac{a+b}{2},\frac{c+d}{2}\right) +f\left( \frac{a+b}{2}%
,c\right) +f\left( \frac{a+b}{2},d\right) }{9}\right.  \notag \\
&&+\frac{f\left( a,c\right) +f\left( b,c\right) +f\left( a,d\right) +f\left(
b,d\right) }{36}  \label{2.5} \\
&&\left. +\frac{1}{\left( b-a\right) \left( d-c\right) }\int_{a}^{b}%
\int_{c}^{d}f\left( x,y\right) dydx-A\right\vert  \notag \\
&\leq &\left( b-a\right) \left( d-c\right)  \notag \\
&&\times \left\Vert \frac{\partial ^{2}f}{\partial t\partial s}\left(
ta+\left( 1-t\right) b,sc+\left( 1-s\right) d\right) \right\Vert _{\infty
}\int_{0}^{1}\int_{0}^{1}\left\vert p\left( x,t\right) q\left( y,s\right)
\right\vert dtds.  \notag
\end{eqnarray}%
By a simple calculation,%
\begin{equation}
\int_{0}^{1}\int_{0}^{1}\left\vert p\left( x,t\right) q\left( y,s\right)
\right\vert dtds=\frac{25}{1296}.  \label{2.6}
\end{equation}%
If we use (\ref{2.6}) in (\ref{2.5}), we have 
\begin{eqnarray*}
&&\left\vert \frac{f\left( a,\frac{c+d}{2}\right) +f\left( b,\frac{c+d}{2}%
\right) +4f\left( \frac{a+b}{2},\frac{c+d}{2}\right) +f\left( \frac{a+b}{2}%
,c\right) +f\left( \frac{a+b}{2},d\right) }{9}\right. \\
&&\left. +\frac{f\left( a,c\right) +f\left( b,c\right) +f\left( a,d\right)
+f\left( b,d\right) }{36}+\frac{1}{\left( b-a\right) \left( d-c\right) }%
\int_{a}^{b}\int_{c}^{d}f\left( x,y\right) dydx-A\right\vert \\
&\leq &\frac{25\left( b-a\right) \left( d-c\right) }{1296}\left\Vert \frac{%
\partial ^{2}f}{\partial t\partial s}\left( ta+\left( 1-t\right) b,sc+\left(
1-s\right) d\right) \right\Vert _{\infty }.
\end{eqnarray*}%
This completes the proof.
\end{proof}

\end{document}